\documentclass[11pt]{article}
\usepackage{latexsym,amssymb}
\usepackage{pifont,mathptmx}
\usepackage[T1]{fontenc}
\usepackage{a4wide}
\usepackage{xcolor}
\usepackage{theorem}
\usepackage{comment}

\title{\sc On the stability of Hamiltonian relative equilibria with non-trivial isotropy}

\author{\sc James Montaldi \& Miguel Rodr\'iguez-Olmos}

\newtheorem{theorem}{Theorem}
\newtheorem{lemma}[theorem]{Lemma}

\newenvironment{proof}%
        {\addvspace\baselineskip\noindent {\sc Proof:}\quad}%
        {\hfill \ding{114} \par\addvspace\baselineskip}  

\newenvironment{proofof}[1]%
        {\addvspace\baselineskip\noindent {\sc Proof of #1:}\quad}%
        {\hfill \ding{114} \par\addvspace\baselineskip}  

{\theorembodyfont{\normalfont}
\newtheorem{example}[theorem]{Example}

}

\newcommand\RR{\mathbb{R}}

\newcommand\fg{\mathfrak{g}}
\newcommand\fh{\mathfrak{h}}
\newcommand\fk{\mathfrak{k}}
\newcommand\fm{\mathfrak{m}}
\newcommand\fn{\mathfrak{n}}
\newcommand\fp{\mathfrak{p}}

\renewcommand\tt{\mathfrak{t}}
\newcommand\so{\mathfrak{so}}

\newcommand{\vv}{\mathbf{v}}

\newcommand\Ad{\mathop\mathsf{Ad}\nolimits}
\newcommand\Coad{\mathop\mathsf{Coad}\nolimits}
\renewcommand\d{\mathsf{d}}
\newcommand\e{\mathsf{e}}
\newcommand\ii{\mathsf{i}}
\newcommand{\half}{{\textstyle\frac12}}
\newcommand{\JJ}{\mathbf{J}}
\newcommand{\SO}{\mathsf{SO}}

\newcommand{\restr}[1]{\,\vrule height1.1ex width.4pt
               depth1.2ex\lower1.0ex\hbox{\scriptsize $\:#1$}}

\newcommand{\todo}[1]{\vspace{3mm}\par \noindent
\framebox{\begin{minipage}[c]{0.95 \textwidth} \raggedright\texttt{\color{red} #1}
\end{minipage}}\vspace{3mm}\par}

\newcommand{\note}[1]{\vspace{3mm}\par \noindent
\framebox{\begin{minipage}[c]{0.95 \textwidth} {\raggedright\color{gray} #1}
\end{minipage}}\vspace{3mm}\par}

\renewcommand{\note}[1]{}

\begin{document}

\maketitle

\centerline{\it Dedicated to the memory of Jerrold E.~Marsden}

\bigskip

\begin{abstract}
We consider Hamiltonian systems with symmetry, and relative equilibria with isotropy subgroup of positive dimension. The stability of such relative equilibria has been studied by Ortega and Ratiu \cite{OR99} and by Lerman and Singer \cite{LS98}. In both papers the authors give sufficient conditions for stability which require first determining a splitting of a subalgebra of $\fg$, with different splittings giving different criteria. In this note we remove this splitting construction and so provide a more general and more easily computed criterion for stability.    The result is also extended to apply to systems whose momentum map is not coadjoint equivariant. 
\end{abstract}

\paragraph{Introduction}
Many Hamiltonian systems arising in nature possess symmetry and in particular continuous symmetry---most commonly a group of rotations or Euclidean motions, whether in the plane or in space.  In this note we consider relative equilibria in such systems, which are motions that coincide with 1-parameter symmetry transformations.  Given such a relative equilibrium, it is often important to decide on its (nonlinear) stability, and there are criteria for determining this based on Dirichlet's criterion for ordinary equilibria, but involving the velocity of the relative equilibrium through an appropriate element of the Lie algebra $\fg$ of the group $G$. 

If the action is locally free at the relative equilibrium (meaning the isotropy subgroup at any point of the relative equilibrium is finite) then the ``relative Dirichlet criterion'' is straightforward because the velocity corresponds to a unique element of the Lie algebra $\fg$.  However, when the action fails to be locally free the story is less clear because there will be many different  ``group velocities'' for a given physical velocity. In the late 1990s, two papers were published, by Ortega and Ratiu \cite{OR99} and Lerman and Singer \cite{LS98}, adapting the Dirichlet criterion to deal with this case, while a paper by the first author \cite{Mo97} provides a more topological criterion, that of an ``extremal relative equilibrium'', which we will use in the proof below.  The method of Ortega-Ratiu and Lerman-Singer involves having a splitting of the Lie algebra $\fg$, or an invariant inner product on $\fg$, and showing there is a unique preferred group velocity relative to this splitting, and then using this preferred velocity to define a relative Dirichlet criterion, analogous to the locally free case.  Moreover, Ortega and Ratiu give an example showing how different choices of splitting may produce different critieria for stability so it may be necessary to consider all possible splittings.

The purpose of this note is to dispense with the splitting construction, and to show that the relative Dirichlet criterion is sufficient to guarantee stability, using \emph{any} group veclocity, not merely those that arise from a splitting. 
In the special case that the relative equilibrium is an equilibrium, the preferred velocity is always zero, regardless of the splitting or inner product, and we give an example at the end of this note showing that it can be necessary to use a non-zero velocity to establish the stability.

Since the proof is based on the idea of an extremal relative equilibrium introduced in \cite{Mo97}, we need a technical assumption: that the momentum isotropy subgroup is compact, rather than it merely being a split subalgebra needed by Ortega-Ratiu and Lerman-Singer.  Of course, if the group $G$ is compact then this is no loss of generality.

\paragraph{Setup and background}
Let $(P,\omega)$ be a connected symplectic manifold with a proper and Hamiltonian action of a Lie group $G$, with momentum map $\JJ:P\to \fg^*$ and an invariant Hamiltonian $h:P\to\RR$.  Recall that a momentum map is a map satisfying the differential condition
\begin{equation}\label{eq:momentum map}
 \left<\d\JJ_p(\vv),\,\xi\right> = \omega(\xi_P(p),\,\vv),
\end{equation}
for all $p\in P,\,\vv\in T_pP,\,\xi\in\fg$, and is therefore uniquely determined up to a constant. By a theorem of Souriau \cite{Souriau} there is an action of $G$ on  $\fg^*$ for which a given momentum map is equivariant, and we will assume that isotropy groups and orbits of points in $\fg^*$ refer to this action.  If $G$ is compact, the momentum map can be chosen so that this action is the coadjoint action \cite{Mo97,ORbook}, but in general it requires an affine modification of the coadjoint action which we denote $\Coad_g^\theta$:
\[\Coad_g^\theta\mu = \Coad_g\mu + \theta(g),\]
where $\theta:G\to\fg^*$ is a cocycle; details are of course in Souraiu's book \cite{Souriau}, see also \cite{ORbook}.  

Throughout, we will be referring to a point $p\in P$ and we will write $H=G_p$, the isotropy subgroup at $p$ for the action on $P$,   $\mu=\JJ(p)$ and $K=G_\mu$, the isotropy subgroup for the modified coadjoint action on $\fg^*$.  We will also be assuming throughout that $K=G_\mu$ is compact, and this means that the momentum map can be chosen so that the cocycle $\theta$ vanishes on $K$, as pointed out in \cite{MT03}.  The Lie algebras of $H$ and $K$ are of course denoted $\fh$ and $\fk$ respectively.
A central ingredient is the \emph{symplectic slice} $N$ at a point $p\in P$, which is defined to be
\[N := \ker\d\JJ(p)/\fk\cdot p\,.\]

A point $p\in P$ is a relative equilibrium if the Hamiltonian vector field at $p$ is tangent to the group orbit, or equivalently if $p$ is a critical point of the augmented Hamiltonian $h_\xi = h- \JJ_\xi$ for some $\xi\in\fg$, where $\JJ_\xi(x) = \left<\JJ(x),\xi\right>$. Such an element $\xi$ is called a velocity of the relative equilibrium.

If $\fh$ is nonzero, then $\d\JJ_\zeta(p)=0$ for all $\zeta\in\fh$. Consequently, if $\xi$ is a velocity at $p$ then so is $\xi+\zeta$.  There are two simple observations about the velocities $\xi$ of a relative equilibrium $p$ with $\JJ(p)=\mu$.  Firstly, by conservation of momentum, $\xi\in\fk$, and secondly by conservation of symmetry $\xi\in\mathrm{Lie}(N_G(H))$, the Lie algebra of the normalizer in $G$ of $H$. Combining these, one deduces that $\xi\in\fn$, where $\fn := \mathrm{Lie}(N_K(H))$ (this fact is implicit in \cite{OR99}). 

If the action in a neighbourhood of a relative equilibrium $p$ is (locally) free, so $\fh=0$, and if $K$ is compact, then there is a well-known criterion for assuring the stability of the relative equilibrium, extending Dirichlet's criterion for the Lyapounov stability of an equilibrium. The criterion is that the restriction $\d^2h_\xi\restr{N}$ of the Hessian to the symplectic slice be definite. In particular, it was shown by Patrick \cite{Pa92} that under this assumption the relative equilibrium is Lyapounov stable relative to the subgroup $K$, which corresponds to the usual definition of Lyapounov stability, but using $K$-invariant neighbourhoods.  The situation where the action on $P$ is free but $K$ fails to be compact, and more generally where $\mu$ is not `split', is considered in \cite{PRW04}.

There are several results in the literature giving criteria for the stability of relative equilibria at points where the action is not locally free (so at points $p$ with $\fh\neq0$).  They all (as do we) require the group action on $P$ to be proper, at least in a neighbourhood of $p$. The criteria of Lerman-Singer \cite{LS98} and Ortega-Ratiu \cite{OR99} begin with requiring an $H$-invariant splitting in \cite{LS98} of the Lie algebra $\fk$, or in \cite{OR99} of $\fn$, as 
\[\fk = \fm \oplus \fh \quad\textrm{or}\quad \fn = \fp\oplus\fh\,.\]
These decompositions are constructed by using an $H$-invariant inner product on $\fk$ or $\fn$, which exist as $H$ is compact, where $H$ acts on $\fk$ or $\fn$ by the adjoint action. The two cases are related by noting that $\fn$ is an invariant subspace under the action by $H$, so any invariant inner product on $\fn$ can be extended to one on $\fk$, while any one on $\fk$ restricts to one on $\fn$, and consequently one can choose $\fp=\fm\cap\fn$.  If $\xi_1\in\fg$ is a velocity of the relative equilibrium then as we have pointed out $\xi_1\in\fn$, and so the set of all velocities is the affine subspace $\xi_1+\fh$ of $\fn$. 

The criterion for stability in both papers is as before that $\d^2h_{\xi^\perp}\restr{N}$ should be a definite quadratic form, where now $\xi^\perp$ is the ``orthogonal velocity'' which is defined to be the uniqe velocity orthogonal to $\fh$ with respect to the chosen splitting, and hence contained in $\fm$ or $\fp$ (and hence always in $\fp$).  Since the inner product (or splitting) is $H$-invariant, the uniqueness of $\xi^\perp$ shows that it is fixed by the adjoint action of $H$.

There is some flexibility in this construction as there may be a choice of invariant inner product, and usually a different choice of inner product leads to a different criterion.  In particular if $\fk$ is Abelian, then any inner product is allowed, and hence any splitting, so if $p$ is not an equilibrium then any velocity can be realized as an orthogonal velocity.  Notice however, that if $p$ is an equilibrium then the orthogonal velocity is always 0, regardless of the splitting (we give an explicit example at the end of this note).  Furthermore, it is not hard to find situations where there is a unique splitting in which case there is again a unique orthogonal velocity; for example if $G=G_1\times G_2$ with $G_1$ semisimple,  $H=G_1$ and $\mu=0$, then $\xi^\perp=(0,\xi_2)$.


\paragraph{Main result} The aim of this note is to eliminate the splitting construction, so to be able to use any velocity for the criterion, not only orthogonal ones.  We show below (Lemma \ref{lemma:well-defined}) that for any velocity $\xi$, the Hessian $\d^2h_\xi(p)$ induces a well-defined quadratic form on the symplectic slice, which we denote $\d^2h_\xi\restr{N}$. Note that in general this quadratic form is not $H$-invariant, although it is if $\xi=\xi^\perp$ for any orthogonal velocity.

\begin{theorem}\label{thm:stability}
Let $(P,\omega)$ be a symplectic manifold with a proper Hamiltonian action of the Lie group $G$, with momentum map $\JJ$, and let the smooth invariant function $h:P\to \RR$ define a symmetric Hamiltonian system on $P$.  Let $p$ be a relative equilibrium of this system, and suppose that $K:=G_\mu$ is compact, where $\mu=\JJ(p)$ and $G$ acts on $\fg^*$ so that $\JJ$ is equivariant.  Let $\xi\in\fg$ be a velocity of $p$.  If the quadratic form $\d^2h_{\xi}\restr{N}$ on the symplectic slice at $p$ is definite then $p$ is Lyapounov stable relative to $K$.
\end{theorem}

Before proving this theorem, we recall the Witt-Artin decomposition of the tangent space $T_pP$ using the group action and symplectic form, and the Marle-Guillemin-Sternberg normal form.  Define four sub-quotients of $T_pP$ as follows,
\begin{equation}
  \label{eq:Witt-Artin}
 \begin{array}{rcl}
  T_0 &=& \fg\cdot p \cap \ker \d\JJ(p)\ =\ \fk\cdot p\,, \\
  T &=& \fg\cdot p/T_0\,, \\
  N &=& \ker \d\JJ(p)/T_0\,, \\
  N_0 &=& T_pM/(\, \fg\cdot p + \ker\d\JJ(p))\,.
\end{array}
\end{equation}
The subspaces $T$ and $N$ are symplectic, while $T_0$ is isotropic, and paired with $N_0$ by the symplectic form.  The group action defines isomorphisms $\fk/\fh\simeq T_0$ and $\fg/\fk\simeq T$, so the symplectic form provides an isomorphism 
\[N_0\simeq (\fk/\fh)^* \simeq \fk^*\cap\fh^\circ \,.\]
Here $\fk^*\cap\fh^\circ$ is the annihilator of $\fh$ within $\fk^*$; some authors denote it by $\mathfrak{m}^*$ (after the splittings described above, though it is independent of the splitting). 

Each space carries an action of the isotropy subgroup $H$, and as this group is compact there is an isomorphism of $H$-spaces
\[T_pP \simeq T_0\oplus T\oplus N\oplus N_0.\]
The vector space $N$ with its symplectic structure and action of $H$ is the symplectic slice and has its own homogeneous quadratic momentum map $\JJ_N:N\to \fh^*$.  

\note{With respect to this decomposition, the symplectic form on $T_pP$ has the form
\[.
  [\omega]_p=\left[\begin{array}{cccc}
      0&0&0&A\cr 0&\omega_{T}&0&*\cr 0&0&\omega_N&*\cr -A^t&*&*&*
    \end{array}\right],
\]
where $\omega_{T}$ and $\omega_N$ are the restrictions of
$\omega$ to $T$ and $N$, $A$ is nondegenerate (and by choosing bases of $T_0$ and $N_0$ can be made to be the identity matrix), and the $*$'s depend on the splitting and are of little interest.  \fbox{do we want this?}}

The Witt-Artin decomposition determines the local geometry of the action, using the Marle-Guillemin-Sternberg normal form.  This states that there is an invariant neighbourhood of $p$ which is $G$-symplectomorphic to an invariant neighbourhood $U$ of $[e,0,0]$ in the symplectic space $Y$ with momemtum map $\JJ_Y:Y\to\fg^*$ given by,
\begin{equation}\label{eq:MGS}
\begin{array}{rcl}
Y&=&G\times_{H} \left((\fk^*\cap\fh^\circ) \times N\right),\\[6pt]
\JJ_Y([g,\,\rho,\,v]) &=& \Coad_g^\theta(\mu+\rho+\JJ_N(v)).
\end{array}
\end{equation}
Here the $H$-action on $G\times\left(\fk^*\cap\fh^\circ\right)\times N$ is by
$h\cdot(g,\,\rho,\,v) = (gh^{-1}, \Coad_h\rho,h\cdot v)$ (recall we have chosen $\theta$ to vanish on $K$, and hence on $H$). The notation $[g,\,\rho,\,v]$ denotes the $H$-orbit of $(g,\,\rho,\,v)$.  The $G$-action is simply $g_1\cdot[g,\,\rho,\,v]=[g_1g,\,\rho,\,v]$.  Since a neighbourhood of $p$ in $P$ is diffeomorphic to $U\subset Y$, the Hamiltonian $h$ on $P$ defines a Hamiltonian on $U$, which we also denote by $h$.

\begin{proofof}{Theorem \ref{thm:stability}}
The proof is in three stages. Firstly we apply the cross-section theorem of Gullemin and Sternberg \cite{GS84}, as modified in \cite{MT03} to deal with non-coadjoint actions. This reduces the problem to a system on a smaller space with compact symmetry group $K$. Secondly, we show that the relative equilibrium is extremal in the sense of \cite{Mo97} and, since $K$ is compact, we deduce that it is Lyapounov stable relative to $K$. And thirdly, we apply a result of Lerman and Singer \cite{LS98} to deduce that the original relative equilibrium is also Lyapounov stable relative to $K$.

\medskip

(1) Since $K$ is compact, there is a $K$-invariant slice $S_\mu$ to the modified coadjoint orbit $G\cdot\mu$.  It is shown in \cite[Section 3]{MT03} that, in a neighbourhood of $p$, $R:=\JJ^{-1}(S_\mu)$ is a $K$-invariant symplectic submanifold of $P$, and that the momentum map $\JJ_R:R\to\fk^*$ can be chosen to be the full momentum map $\JJ$ restricted to $R$ followed by the natural projection $\fg^*\to\fk^*$.  Furthermore, as $R$ is a union of level sets of $\JJ$, it is invariant under the original dynamics, so the restriction of $h$ to $R$ determines the dynamics on $R$ by the usual equations of Hamilton.  Note that the Witt-Artin decomposition of $T_pR$ is
\[T_pR = T_0 \oplus N\oplus N_0\]
where $T_0,\,N$ and $N_0$ are the same spaces as in eq.~(\ref{eq:Witt-Artin}), and $\ker\d\JJ_R(p)=T_0\oplus N$.

\medskip

(2) We want to show that the relative equilibrium $p$ is extremal \cite{Mo97}; this means that the corresponding equilibrium point in the reduced space is a local extremum of the reduced Hamiltonian.  Let us assume that $\d^2h_\xi\restr{N}$ is positive definite; the negative definite case proceeds similarly.  We therefore want to show there is a neighbourhood $U$ of $p$ in $\JJ^{-1}(\mu)=\JJ_R^{-1}(\mu)$ such that $x\in U\setminus K\cdot p \Rightarrow h(x) > h(p)$.  

To this end we use the Marle-Guillemin-Sternberg normal form (\ref{eq:MGS}), with $G$ replaced by $K$, thus:
\begin{equation}\label{eq:MGS-mu}
\begin{array}{rcl}
Z&=&K\times_{H} \left(\fh^\circ \times N\right),\\[6pt]
\JJ_Z([g,\,\rho,\,v]) &=& \mu+\Coad_g(\rho+\JJ_N(v)),
\end{array}
\end{equation}
where $\JJ_Z:Z\to\fk^*$. (Since $g\in K$ we have $\Coad_g(\mu)=\mu$, and we now take $\fh^\circ$ to mean the appropriate subset of $\fk^*$.)

From eq.~(\ref{eq:MGS-mu}), one sees that $\JJ_Z^{-1}(\mu) = \{[g,\,\rho,\,v]\in Z \mid \rho=0,\;\JJ_N(v)=0\}$ (this can also be found in \cite[Proposition 13]{BL97}), so that the reduced space at $\mu$ in this model is 
\[Z_\mu \ = \ \JJ_Z^{-1}(\mu)/K \ \simeq\ \ \JJ_N^{-1}(0)/H.\]
Now let $\bar q\in Z_\mu$ be a point distinct from $\bar p$, and let $q\in N$ be a corresponding point in $\JJ_N^{-1}(0)$. We wish to show $h_\xi(q)>h_\xi(p)$. 

We are assuming $\d^2h_\xi\restr{N}$ to be non-degenerate.  By the Morse Lemma there is therefore a diffeomorphism $\phi$ of $N$ (preserving $p$) such that $h_\xi = h_\xi(p)+\d^2h_\xi\restr{N}\circ\phi$. It follows, using the fact that $\d^2h_\xi\restr{N}$ is positive definite, that $h_\xi(q) = h_\xi(p) + \d^2h_\xi\restr{N}(\phi(q)) > h_\xi(p)$, as required.

\medskip

(3) Finally, Lerman and Singer \cite[Proposition 2.3]{LS98} states that if $p$ is a relative equilibrium on $R$ which is Lyapounov stable relative to $K$, then $p$ is also Lyapounov stable relative to $K$ in the full $G$-invariant system on $P$.
\end{proofof}

\begin{lemma} \label{lemma:well-defined}
If $p$ is a relative equilibrium and $\xi\in\fg$ is a velocity, then the restriction of the quadratic form $\d^2h_\xi(p)$ to $\ker\d\JJ$ descends to a well-defined quadratic form on the symplectic slice $N$.
\end{lemma}

\begin{proof}
Since the Hamiltonian vector field is equivariant, the point $g\cdot p$ is a relative equilibrium with velocity $\Ad_g\xi$, and hence for all $g\in G$, the corresponding differential vanishes:
\[\d h_{\Ad_g\xi}(g\cdot p) = 0.\]
Write $g=\exp(t\eta)$ for $\eta\in\fg$, and differentiate with respect to $t$ at $t=0$ to obtain
\[\d^2h_\xi(\eta\cdot p,-) -\d\JJ_{[\eta,\xi]} = 0,\]
where the differentials are taken at $p$. It follows that for any $v\in \ker\d\JJ(p)$ we have
\[\d^2h_\xi(\eta\cdot p,\,v)=0.\]
It then follows that, given any $\eta\in\fk$ (so that $\eta\cdot p\in\ker\d\JJ(p)$) 
\[\d^2h_\xi(v+\eta\cdot p,\,v+\eta\cdot p)=\d^2h_\xi(v,\,v),\]
as required.
\end{proof}

\begin{example}\label{eg}
Consider $P=\RR^4$ with symplectic form $\omega=\d x_1\wedge\d y_1 + \d x_2\wedge \d y_2$,  and with a Hamiltonian action of $\SO(2)$ given by
\[\theta\mapsto \pmatrix{R_\theta &0 \cr 0 & R_{-\theta}}.\]
The momentum map is $\JJ(x_1,y_1,x_2,y_2) = \half(x_1^2+y_1^2) - \half(x_2^2+y_2^2)$, and the origin is a fixed point of the action.  Consider the $\SO(2)$-invariant Hamiltonian 
\[h = (x_1^2+y_1^2) - 2(x_2^2+y_2^2) + \cdots,\]
where the $\cdots$ refers to higher order terms invariant under $\SO(2)$. The origin is a (relative) equilibrium with velocity 0, but also velocity equal to any $\xi\in\so(2) \simeq\RR$.  It is easy to see that $\d^2h_\xi$ is negative definite for $\xi\in(2,\,4)$, so that Theorem \ref{thm:stability} guarantees the origin is Lyapounov stable (relative to $\SO(2)$). However, as expained earlier, the criteria of Lerman-Singer and Ortega-Ratiu do not guarantee stability as the unique orthogonal velocity is 0 and $\d^2h(0)$ is not definite.
\end{example}

We give a different proof of Theorem \ref{thm:stability} in \cite{MRO}, using the so-called bundle equations, and present a different example.

\paragraph{Acknowledgements} 
The research of M.R-O. was supported by the research project MTM2006-03322 and a European Marie Curie Fellowship (IEF)  held at the University of Manchester.

\small

\bigskip

{\obeylines \setlength{\parindent}{0pt} \it
School of Mathematics 
University of Manchester
Manchester M13 9PL 
UK 

\bigskip

Department of Applied Mathematics IV
Technical University of Catalonia 
Barcelona
Spain}

\end{document}